\documentclass[a4paper,11pt]{amsart}

\usepackage[english]{my-shortcuts}
\usepackage{srcltx,algorithm,algorithmic}
\usepackage{a4wide}

\usepackage{enumitem}

\usepackage{tikz}
\begin{document}

	\title[On the subdifferential of a function of the tensor spectrum]
	{On the subdifferential of symmetric convex functions of the spectrum for symmetric and orthogonally decomposable tensors}
	\author{St\'ephane Chr\'etien {\tiny and} Tianwen Wei}

	\address{St\'ephane Chr\'etien is with the National Physical Laboratory,
		Mathematics and Modelling, \\
		Hampton Road,
		Teddington, TW11 OLW, UK}
	\email{stephane.chretien@npl.co.uk}
	
	\address{Tianwen Wei is with the Department of Mathematics and Statistics \\
		Zhongnan University of Economics and Law \\
		Nanhu Avenue 182
		430073, Wuhan, China}
	\email{tianwen.wei.2014@ieee.org}
	
	\begin{abstract}
		The subdifferential of convex functions of the singular spectrum of real matrices has been widely studied in matrix analysis, optimization and automatic control theory.
		Convex optimization over spaces of tensors is now gaining much interest due to its potential applications in signal processing, statistics and engineering.
		The goal of this paper is to present an extension of the approach by Lewis \cite{lewis1995convex} for the analysis of the subdifferential of certain convex functions of the
		spectrum of symmetric tensors. We give a complete characterization of
		the subdifferential of Schatten-type tensor norms for symmetric tensors.
		Some partial results in this direction are also given for Orthogonally Decomposable tensors.
	\end{abstract}

	\maketitle
	

	
	\tableofcontents

	\section{Introduction}
	
	\subsection{Background}
    Multidimensional arrays, also known as tensors are higher-order generalizations of vectors and matrices.
	In recent years, they have been the subject of extensive interest in various extremely active fields such as e.g. statistics, signal processing, automatic control, etc \ldots where a lot of problems involve quantities that are intrinsically multidimensional such as higher order moment tensors \cite{anandkumar2014tensor}. Many natural and useful quantities in linear algebra such as the rank or the Singular Value Decomposition turn out to be very difficult to compute or generalize in the tensor setting \cite{Kolda2009,landsberg2012tensors, hackbusch2012tensor}.
	Fortunately, efficient approaches exist in the case of
	symmetric tensors which lie at the heart of the moment approach which recently proved very efficient for addressing essential problems in Statistics/Machine Learning such as Clustering, estimation in Hidden Markov Chains, etc \ldots See the very influencial paper \cite{anandkumar2014tensor} for more details. In many statistical models such as the ones presented in \cite{anandkumar2014tensor}, the rank of the involved is
	low and one expects that the theory of sparse recovery can be applied to recover them form just a few observations just as in the case of Matrix Completion \cite{candes2009exact}, \cite{Candes2010tao} Robust PCA \cite{Candes2011robustpca} and Matrix Compressed Sensing \cite{recht2010guaranteed}.
	In such approaches to Machine Learning, one usually have to solve a penalized least squares problem of the type
	\begin{align*}
	    & \min_{X \in \mathbb R^{n_1 \times n_2}}
	    \quad \Vert y - \mathcal A(X) \Vert + \lambda \ p(X),
	\end{align*}
	where the penalization $p$ is rank-sparsity promoting such as the nuclear norm and $\mathcal A$ is a linear operator taking values in $\mathbb R^n$. In the tensor setting, we look for solutions of problems of the type
	\begin{align*}
	    & \min_{X \in \mathbb R^{n_1 \times \cdots \times n_D}}
	    \quad \Vert y - \mathcal A(X) \Vert + \lambda \ p(X),
	\end{align*}
	for $D>2$ and $p$ is a generalization of the nuclear norm or some Schatten-type norm for tensors. The extention of Schatten norms to the tensor setting has to be carefully defined. In particular, several nuclear norms can be naturally defined \cite{yuan2014tensor}, \cite{gandy2011tensor}, \cite{mu2014square}. Moreover, the study of the efficiency of sparsity promoting penalization relies crucially on the knowledge of the subdifferential of the norm involved as achieved in \cite{amelunxen2014living} or \cite{lecue2016regularization}, or at least a good approximation of this subdifferential \cite{yuan2014tensor}
	\cite{mu2014square}.
    In the matrix setting, the works of
    \cite{watson1992characterization, lewis1995convex} are famous for providing a neat characterization of the subdifferential of matrix norms or more generaly functions of the matrix enjoying enough symmetries.
    In the 3D or higher dimensional setting, however, the
	case is much less understood. The relationship between
	the tensor norms and the norms of the flattenings are intricate
	although some good bounds relating one to the other can be obtained
	as in \cite{hu2015relations}. Notice that many recent works use the nuclear norms
	of the flattenings of the tensors to be optimized, especially in the field of compressed
	sensing; see e.g. \cite{mu2014square, gandy2011tensor}.
	One noticeable exception is the recent work \cite{yuan2014tensor} where a study of the subdifferential of a purely tensorial nuclear norm is proposed. However, in \cite{yuan2014tensor}, only a subset of the subdifferential is given but the subdifferential itself could not be fully characterized.


	Our goal in the present paper is to extend previous results on matrix norms to the tensor setting.
    The focus will be on two special type of tensors, namely symmetric tensors and orthogonally decomposable tensors (abbreviated as {\em odeco} tensor hereafter). Symmetric tensors are invariant under any permutation of its indices \cite{comon2008symmetric}. They
    play a important role in many applications, e.g. Gaussian Mixture Models (GMM), Independent Component
	Analysis (ICA) and Hidden Markov Models (HMM), see \cite{anandkumar2014tensor} for a survey.
    {\em odeco} tensors have a diagonal core in their Higher Order Singular Value Decomposition (HOSVD) \cite{lathauwer2000HOSVD}. They are special structured tensors that inherit many nice properties of their matrix counterpart. In this contribution, we propose a general study the subdifferential of certain convex functions of the spectrum of  these  tensors and apply our results to the
	computation of the subdifferential of useful and natural tensor norms.
	The convex conjugate approach used in this paper stems from the elegant work of \cite{lewis1995convex}.
	One key ingredient for the understanding of
	tensor norms is the tensor version Von Neumann's trace inequality and the precise description of the equality case. 
We suspect that
	the lack of results on the subdifferential of tensor norms in the literature is due to the fact that an extension of the Von Neumann inequality for tensors did not exist until recently; see \cite{chretien2015neumann}. %

	The plan of the paper is as follows. In Section \ref{gentens}, we provide a general overview of tensors and their spectral factorizations. In Section \ref{subdf}, we provide a general formula for the subdifferential for symmetric tensors and {\em odeco} tensors.  Finally, in Section \ref{applischatten}, we
	provide formulas for the subdifferential of Schatten norms for symmetric tensors and the subset of {\em odeco} tensors in the subdifferential of Schatten norms for {\em odeco} tensors.

	\subsection{Notations}
	\subsubsection{Convex functions}
	For any convex function $f$ : $\mathbb R^{n}\mapsto \mathbb R\cup \{+\infty\}$, the conjugate function $f^*$ associated to $f$ is defined by
	\bean
	f^*(g) \defeq \sup_{x \in \mathbb R^n}\quad \langle g, x  \rangle - f(x).
	\eean
	The subdifferential of $f$ at $x \in \mathbb R^n$ is defined by
	\bean
	\partial f & \defeq & \left\{ g \in \mathbb R^n \mid \forall y, \in \mathbb R^n \hspace{.3cm}
	f(y)\ge f(x)+\langle g,y-x\rangle  \right\}.
	\eean
	It is well known (see e.g. \cite{hiriart2013convex}) that
	$g\in\partial f(x)$ if and only if
	\bean
	f(x) + f^*(g) & = &  \langle g, x\rangle.
	\eean
	
	\subsubsection{Tensors}
	Let $D$ and $n_1,\ldots,n_D$ be positive integers.
	In the present paper,  a multi-dimensional array $\Xc$ in $\mathbb R^{n_1\times \cdots\times n_D}$ is called a $D$-mode tensor.  If $n_1=\cdots=n_D$, then we will say that tensor $\Xc$ is cubic.
	The set of $D$-mode cubic tensors
	will be denoted by $\Rb^{n\times \cdots \times n}$ or $\Rb^{nD}$ with a slight abuse of notation.

	The mode-$d$ fibers of a tensor $\Xc$ are the vectors obtained by varying the index $i_d$ while keeping the other indices fixed.

	It is often convenient to rearrange the elements of a tensor so that they form a matrix. This operation is referred to as matricization and can be defined in different ways. In this work, $\Xc_{(d)}$ stands for a
	matrix in $\Rb^{n_d\times \prod_{i=1;i\neq d}^{D}n_i}$ obtained by stacking the mode-$d$ fibers of $\Xc$ one after another with forward cyclic ordering \cite{lathauwer2010}.  Inversely, we
	define the tensorization operator $\Ts^{(d)}$ as the adjoint of the mode-$d$ matricization operator, i.e. it is such that
	\bean
	\langle \Xc_{(d)},M \rangle & = & \langle \Xc,\Ts^{(d)}(M) \rangle
	\eean
	for all $\Xc\in \mathbb R^{n_1\times \cdots \times n_D}$ and all $M\in \Rb^{n_d\times \prod_{i=1;i\neq d}^{D}n_i}$, where $\langle \cdot, \cdot\rangle$ denotes the scalar product defined in Section \ref{tensornorm}.
	
	The mode-$d$ multiplication of a tensor $\Xc \in \mathbb R^{n_1 \times \cdots \times n_D}$ by a matrix $M \in \mathbb R^{n'_d \times n_d}$, denoted by
	$\Xc \times_d M$, yields a tensor in $\mathbb R^{n_1\times \cdots \times n_{d-1} \times n^\prime_d \times n_{d+1} \times \cdots n_d}$.
	It is defined by
	\bean
	\left(\Xc \times_d M \right)_{i_1,\ldots,i_D} & = & \sum_{i_d} \quad \Xc_{i_1,\ldots,i_{d-1},i_d,i_{d+1},\ldots,i_D}
	M_{i_d,i^\prime_d}.
	\eean

	Last, we denote by $\otimes$ the tensor product, i.e. for any $v^{(1)}, \ldots,v^{(D)}$ with $v^{(d)} \in \mathbb \Rb^{n_d}$, $v^{(1)} \otimes \cdots \otimes v^{(D)}$ is a tensor in
	$\Rb^{n_1\times\cdots\times n_D}$ whose entries are given by
	\bean
	\left(v^{(1)} \otimes \cdots \otimes v^{(D)} \right)_{i_1,\ldots,i_D} & = & v^{(1)}_{i_1} \cdots v^{(D)}_{i_D}.
	\eean

	\section{Basics on tensors}
	\label{gentens}
	
	\subsection{General tensors}
	\subsubsection{Tensor rank}
	If a tensor $\Xc$ can be written as
	\bean
	\Xc=v^{(1)} \otimes \cdots \otimes v^{(D)},
	\eean
	then we say $\Xc$ is a rank one tensor.
	Any tensor $\Xc$ can easily be written as a sum of rank one tensors. Indeed, if $(e_i^{(n)})_{i=1,\ldots,n}$ denotes
	the canonical basis of $\mathbb R^n$, we have
	\bean
	\Xc & = & \sum_{i_1=1,\ldots,n_1,\ldots,i_D=1,\ldots,n_D} \quad x_{i_1,\ldots,i_D} \cdot
	e_{i_1}^{(n_1)} \otimes \ldots \otimes e_{i_D}^{(n_D)},
	\eean
	Among all possible decomposition as a sum of rank one tensors, one may look for the one involving
	the least possible number of summands, i.e.
	\bea
	\label{soro}
	\Xc & = & \sum_{j=1,\ldots,r} \quad  v_{j}^{(1)} \otimes \ldots \otimes v_{j}^{(D)},
	\eea
	for some vectors $v_{j}^{(d)}$, $j=1,\ldots,r$ and $d=1,\ldots,D$. The number $r$ is called the rank of $\Xc$.
	It is already known that the rank of a tensor is NP-hard to compute \cite{landsberg2012tensors}.	
	
	\subsubsection{The Higher Order SVD}
	
	One of the main problems with the ``sum-of-rank-one'' decomposition (\ref{soro}) is that the vectors $v_j^{(d)}$,
	$j=1,\ldots,r$ may not form an orthonormal family of vectors.
	The Tucker decomposition of a tensor is another decomposition which  reveal a possibly smaller tensor $\Sc$ hidden inside $\Xc$ under orthogonal transformations. More precisely,
	we have
	\bea
	\label{Tucker}
	\Xc & = & \Sc(\Xc) \times_1 U^{(1)} \times_2 U^{(2)} \cdots \times_D U^{(D)},
	\eea
	where each $U^{(d)}$ is orthogonal and $\Sc(\Xc)$ is a tensor of the same size as $\Xc$ defined as follows.
	Take the (usual) SVD of the matrix $\Xc_{(d)}$
	\bean
	\Xc_{(d)} & = & U^{(d)} \Sigma^{(d)} {V^{(d)}}^t
	\eean
	and based on \cite{lathauwer2000HOSVD}, we can set
	\bean
	\Sc(\Xc)_{(d)} & = & \Sigma^{(d)} {V^{(d)}}^t\left(U^{(d+1)} \otimes \cdots \otimes U^{(D)} \otimes U^{(1)}
	\otimes \cdots \otimes U^{(d-1)}\right).
	\eean
	Then, $\Sc(\Xc)_{(d)}$ is the mode-$d$ matricization of $\Sc(\Xc)$. One proceeds similarly for all $d=1,\ldots,D$
	and one recovers the orthogonal matrices $U^{(1)},\ldots,U^{(D)}$ which allow us to decompose $\Xc$ as
	in (\ref{Tucker}).
Notice that this construction implies that $\Sc(\Xc)$ has orthonormal fibers for every modes.

	\subsubsection{Norms of tensors\label{tensornorm}}
	Several tensor norms can be defined on the tensor space $\Rb^{n_1\times \cdots \times n_D}$. 
    The first one is a natural extension of the
	Frobenius norm or Hilbert-Schmidt norm from matrices. We start by defining the scalar product
	on $\Rb^{n_1\times \cdots \times n_D}$ as
	\bean
	\langle \Xc, \Yc\rangle & \defeq & \sum_{i_1=1}^{n_1}\cdots  \sum_{i_D=1}^{n_D} x_{i_1,\ldots,i_D} y_{i_1,\ldots,i_D}.
	\eean
	Using this scalar product, we can define the Frobenius norm of tensors as
	\bean
	\|\Xc\|_F & \defeq & \sqrt{\langle \Xc,\Xc\rangle}.
	\eean
    In this work, we shall focus on a family of tensor norms called Schatten-$(p,q)$ norms. The Schatten-$(p,q)$ norm of $\Xc$ is defined by
    \bea\label{schattenpq612}
    \|\Xc\|_{p,q} \defeq \lambda \Big( \sum_{d=1}^D \|\sigma^{(d)}(\Xc)\|_p^q \Big)^{1/q},
    \eea
	where $\sigma^{(d)}(\Xc)$ is the vector of singular values of $\Xc_{(d)}$, called the mode-$d$ spectrum of $\Xc$, and $\lambda$ is a positive constant. 
In the particular case that $p=q=1$ and $\lambda=1/D$, the Schatten-$(1,1)$ norm will be referred to as the \emph{nuclear norm}, and will be denoted by $\|\cdot\|_*$ instead, i.e.
\bean
 \|\Xc\|_* \defeq \frac{1}{D} \sum_{d=1}^D \|\sigma^{(d)}(\Xc)\|_1.
\eean

	\subsection{Orthogonally decomposable tensors}
	\label{odec}
	\begin{defi}
		Let $\Xc$ be a tensor in $\Rb^{n_1\times \cdots \times n_D}$. If
		\bea\label{defi1}\label{odec1}
		\Xc & = & \sum_{i=1}^{r}\alpha_i\cdot u_i^{(1)}\otimes\cdots \otimes u_i^{(D)},
		\eea
		where $r\leq n_1 \wedge \cdots \wedge n_D$, $\alpha_1\geq \cdots \geq \alpha_r > 0$  and
		$\{u_1^{(d)}, \ldots, u_r^{(d)}\}$  is a family of orthonormal vectors for each $d=1,\ldots,D$, then we say
		(\ref{defi1}) is an orthogonal decomposition of $\Xc$ and $\Xc$ is an orthogonally decomposable ({\em odeco}) tensor.
	\end{defi}
	Denote $\alpha=(\alpha_1,\ldots, \alpha_r, 0,\ldots, 0)$ in $\mathbb R^{n_1 \wedge \cdots \wedge n_D}$.
	For each $d\in\{1,\ldots, D\}$, we may complete $\{u_1^{(d)}, \ldots, u_r^{(d)}\}$ with
	$\{u_{r+1}^{(d)},\ldots, w_{n_d}^{(d)}\}$ so that matrix
	$U^{(d)}=(u_1^{(d)},\ldots, u_{n_d}^{(d)})\in\Rb^{n_d\times n_d}$ is orthogonal.
	Using $U^{(1)},\ldots, U^{(D)}$,
	we may write (\ref{defi1}) as
	\bea
	\Xc \label{odec2}
	& = & \diag(\alpha)\times_1 U^{(1)} \times_2 U^{(2)} \cdots\times_D U^{(D)}.
	\eea
	
	\subsection{Symmetric tensors}
	Let $\mathfrak S_D$ be the set of permutations over $\{1,\ldots,D\}$. A $D$-mode cubic tensor $\Xc \in \mathbb R^{nD}$ will be said
	symmetric if for all
	$\tau \in \mathfrak S_D$,
	\bean
	\Xc_{i_1,\ldots,i_D} & = & \Xc_{\tau(i_1),\ldots,\tau(i_D)}
	\eean
	The set of all symmetric tensors of order $n$ will be denoted by $\mathbb S_n$.
	An immediate result is the following useful proposition whose proof is straightforward.
	
	\subsection{Spectrum of tensors \label{image}}
	
	Let $\mathbb E$ denote a subspace of the space of all tensors. Let us define the spectrum as the mapping which to any tensor $\Xc\in \mathbb E$
	associates the vector $\sigma_{\mathbb E}(\Xc)$ given by
	\bean
	\sigma_{\mathbb E}(\Xc) & \defeq & \frac1{\sqrt{D}} \: (\sigma^{(1)}(\Xc),\ldots,\sigma^{(D)}(\Xc)).
	\eean
    Here we stress that the underlying tensor subspace $\Eb$ does make a difference. For instance,
    although $\sigma_{\Rb^{nD}}(\Xc) = \sigma_{\Sb_n}(\Xc)$ for all $\Xc\in\Sb_n$, the two different tensor space $\Rb^{nD}$ and $\Sb$ may result in different subdifferential of the same tensor norm.

	The following result is straight forward.
	\begin{prop}\label{symm}
		If $\Xc$ is either {\em odeco} or symmetric, then $\sigma^{(1)}(\Xc)=\cdots=\sigma^{(D)}(\Xc)$.
	\end{prop}	
	
	\section{Further results on the spectrum}
	
	In this section, we will present some further results on the spectrum such as the question of characterizing the
	image of the spectrum and the subdifferential of a function of the spectrum.
	
	\subsection{The Von Neumann inequality for tensors}
	Von Neumann's inequality says that for any two matrices $X$ and $Y$ in $\mathbb R^{n_1\times n_2}$, we have
	\bean
	\langle X,Y\rangle & \le & \langle \sigma(X),\sigma(Y)\rangle,
	\eean
	with equality when the singular vectors of $X$ and $Y$ are equal, up to permutations when the singular values have multiplicity greater than one. This result has proved useful for the study of the subdifferential of unitarily invariant convex functions of the spectrum in the matrix case in \cite{lewis1995convex}.
	In order to study the subdifferential of the norms of symmetric tensors, we will need a generalization of this result to higher orders.
	This  was worked out in \cite{chretien2015neumann}.
	\begin{defi}
		We say that a tensor $\Sc$ is blockwise decomposable if there exists an integer $B$ and if, for all $d=1,\ldots,D$, there exists a partition $I_1^{(d)} \cup \ldots \cup I_B^{(d)}$ into
		disjoint index subsets of $\{1,\ldots,n_d\}$, such that $\Xc_{i_1,\ldots,i_D}=0$ if for all $b=1,\ldots,B$, $(i_1,\ldots,i_D) \not \in I_b^{(1)}\times \ldots \times I_b^{(D)}$.
	\end{defi}
	An illustration of this block decomposition can be found in Figure \ref{blocks}.
	The following result is a generalization of von Neumann's inequality from matrices to tensors. It is proved in \cite{chretien2015neumann}.
	\begin{theo}\label{VN3}
		Let $\Xc,\Yc\in \Rb^{n_1\times \cdots \times n_D}$ be tensors.
		Then for all $d=1,\ldots,D$, we have
		\bea\label{vonneumann1}
		\langle \Xc, \Yc\rangle \leq \langle \sigma^{(d)}(\Xc), \sigma^{(d)}(\Yc) \rangle.
		\eea
		Equality in (\ref{vonneumann1}) holds simultaneously for all $d=1,\ldots,D$ if and only there
		exist orthogonal matrices $W^{(d)}\in\Rb^{n_d\times n_d}$ for $d=1,\ldots, D$
		and tensors $\Dc(\Xc),\Dc(\Yc)\in\Rb^{n_1\times\cdots\times n_D}$ such that
		\bean
		\Xc&=& \Dc(\Xc) \times_1 W^{(1)} \cdots \times_{D} W^{(D)}, \\
		\Yc&=& \Dc(\Yc) \times_1 W^{(1)} \cdots \times_{D} W^{(D)},
		\eean
		where $\Dc(\Xc)$ and $\Dc(\Yc)$ satisfy the following properties:
		\begin{enumerate}
			\item[(i)] $\Dc(\Xc)$ and $\Dc(\Yc)$ are block-wise decomposable with the same number of blocks, which we will denote by $B$,
			\item[(ii)] the blocks $\{\Dc_b(\Xc)\}_{b=1,\ldots,B}$ (resp. $\{\Dc_b(\Yc)\}_{b=1,\ldots,B}$) on the diagonal of $\Dc(\Xc)$ (resp. $\Dc(\Yc)$) have the same sizes,
			\item[(iii)] for each $b=1,\ldots, B$ the two blocks $\Dc_b(\Xc)$ and $\Dc_b(\Yc)$ are proportional.
		\end{enumerate}
	\end{theo}
	
	\begin{figure}
		\begin{center}\begin{tikzpicture}[scale=0.70]
			
			\draw[thick] (0,0) -- (0,8) -- (8,8)--(8,0)--(0,0);
			\draw[thick] (8,0) -- (12,2) -- (12,10) -- (4,10) -- (0,8);
			\draw[thick] (8,8) -- (12,10);
			\draw[dashed] (4,2) -- (12,2);
			\draw[dashed] (4,2) -- (4,10);
			\draw[dashed] (4,2) -- (0,0);
			
			\draw (0,7) -- (2,7) -- (2,8) -- (3, 8.5);
			\draw (1, 8.5) -- (3, 8.5) -- (3, 7.5) -- (2, 7);
			
			
			\draw (3, 7.5) -- (3, 4.5) -- (4.5,4.5) -- (4.5, 7.5) -- (3, 7.5);
			\draw (3, 7.5) -- (4.5, 8.25) -- (6, 8.25) -- (6, 5.25) -- (4.5, 4.5);
			\draw (4.5, 7.5) -- (6, 8.25);
			
			
			\draw (6, 5.25) -- (6, 4.25) -- (7, 4.25) -- (7, 5.25) -- (6, 5.25);
			\draw (7.5, 4.5) -- (7.5, 5.5) -- (6.5, 5.5);
			
			\draw (6, 5.25) -- (6.5, 5.5);
			\draw (7, 5.25) -- (7.5, 5.5);
			\draw (7, 4.25) -- (7.5, 4.5);
			
			
			\draw (7.5, 4.5) -- (7.5, 4) -- (10, 4) -- (10, 4.5) -- (7.5, 4.5);
			\draw (7.5, 4.5) -- (7.75, 4.625) -- (10.25, 4.625)-- (10.25, 4.125)-- (10, 4);
			\draw (10.25, 4.625) -- (10, 4.5);

			\draw (10.25, 4.125) -- (10.25, 1.625) -- (11.25, 1.625) -- (11.25, 4.125) -- (10.25, 4.125);
			\draw (10.25, 4.125) -- (11, 4.5) -- (12, 4.5) -- (11.25, 4.125);
			\end{tikzpicture}
		\end{center}
		\caption{A block-wise diagonal tensor.}
		\label{blocks}
	\end{figure}
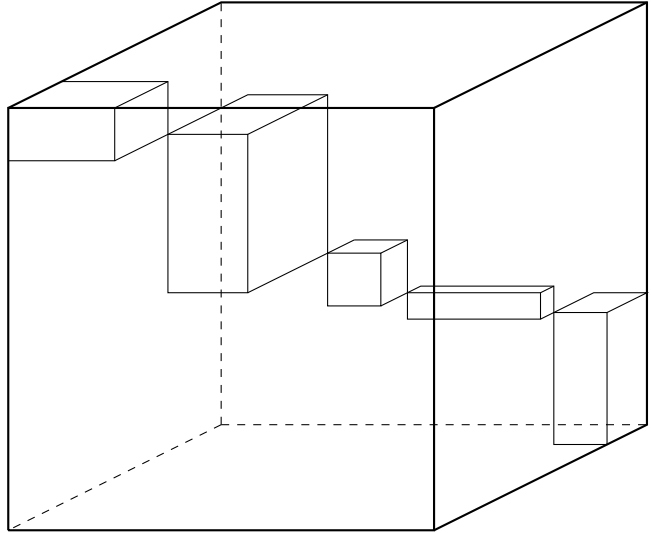
	
	\subsection{Surjectivity of the spectrum}
	Note that any diagonal tensor is both {\em odeco} and symmetric. A diagonal tensor $\Xc$ with non-negative diagonal entries $(\lambda_1,\ldots,\lambda_n)$ satisfies
	$\sigma^{(1)}(\Xc)=\cdots=\sigma^{(D)}(\Xc)=(\lambda_1,\ldots,\lambda_n)^\TR$. Therefore, for any non-negative vector $s$, there exist a symmetric and {\em odeco} tensor $\Xc$ such that
	$\sigma(\Xc)=(s,\ldots,s)$.
	
	Notice that general tensors have different spectra along all the different modes and the question of analysing the surjectivity is much more subtle.
	
	
	\section{The subdifferential of functions of the spectrum}
	\label{subdf}
	In this section, we present a characterization of the subdifferential of a convex function
	$f: \Rb^n\times \cdots \times \Rb^n\mapsto\Rb$ of the spectrum for cubical tensors on symmetric and {\em odeco} tensors.

	\subsection{Lewis' characterization of the subdifferential}	
	\label{subcharac}
	Let us recall that the spectrum is
	defined on a subspace $\mathcal E$.
	In this section, we recall the result of Lewis in the setting of tensor spectra, which characterizes the
	subdifferential if the formula
	\begin{align}
	(f \circ \sigma_{\mathbb E})^*=f^* \circ \sigma_{\mathbb E}
	\label{Conjformula}
	\end{align}
	holds on the domain of definition of $\sigma_{\mathbb E}$. The proof is exactly the same as in \cite{lewis1995convex} thanks to the tensor version of Von Neumann's inequality.  We recall it here for the sake of completeness.
	\begin{theo}\label{119a}
		Let $f: \Rb^{n}\times\cdots\times\Rb^n \mapsto \Rb$ be a convex function. Let $\Xc$ and $\Yc$ be two tensors
		in $\mathbb R^{n D}$. If \eqref{Conjformula} holds,
		then  $\Yc\in\partial (f\circ\sigma_{\mathbb E})(\Xc)$ if and only if the following conditions are satisfied:
		\begin{enumerate}
			\item[(i)] $\sigma_{\mathbb E} (\Yc) \in\partial f(\sigma_{\mathbb E}(\Xc))$,
			\item[(ii)] equality holds in the Von Neumann inequality, i.e. $\langle \Xc, \Yc\rangle = \langle \sigma_{\mathbb E}(\Xc), \sigma_{\mathbb E}(\Yc)\rangle$.
		\end{enumerate}
	\end{theo}
	\begin{proof}
		As is well known, $\Yc\in\partial (f\circ\sigma)(\Xc)$ if and only if
		\begin{align*}
		(f\circ\sigma_{\mathbb E})(\Xc) + (f\circ \sigma_{\mathbb E})^*(\Yc) & = \langle \Xc, \Yc\rangle.
		\end{align*}
		Recall that we assumed the following equality to hold
		\begin{align*}
		(f\circ\sigma_{\mathbb E})(\Xc) + (f\circ \sigma_{\mathbb E})^*(\Yc) & = f(\sigma_{\mathbb E}(\Xc)) + f^*
		(\sigma_{\mathbb E}(\Yc)).
		\end{align*}
		On the other hand,
		\begin{align*}
		f(\sigma_{\mathbb E}(\Xc)) + f^*(\sigma_{\mathbb E}(\Yc)) & \geq \langle \sigma_{\mathbb E}(\Xc), \sigma_{\mathbb E}(\Yc) \rangle,
		\end{align*}
		where equality takes place if and only if
		\begin{align*}
		\sigma_{\mathbb E}(\Yc) \in\partial f (\sigma_{\mathbb E}(\Xc)).
		\end{align*}
		Finally, by the von Neumann inequality, we have
		\begin{align*}
		\langle \sigma_{\mathbb E}(\Xc), \sigma_{\mathbb E}(\Yc) \rangle
		\ge \langle  \Xc, \Yc \rangle,
		\end{align*}
		where the equality takes place if and only if the equality condition is satisfied.
	\end{proof}

	\subsection{The symmetric case}
	Throughout this section $\mathbb E$
	will be the set $\mathbb S_n$ of all symmetric
	tensors in $\mathbb R^{n D}$.
	\subsubsection{Proving \eqref{Conjformula} for symmetric tensors}
	There exists a simple formula for the conjugate of the composition of the spectrum with a
	convex function. This formula will be helpful for gaining useful information on the
	subdifferential of convex functions of the spectrum.
	\begin{theo} \label{conjugate1}
		Let $f: \Rb^{n}\times\cdots\times\Rb^n \mapsto \Rb$ be a convex function. Let $\Xc$ be a {\em symmetric} tensor in $\mathbb R^{n D}$. Then,
		\bean
		(f\circ\sigma_{\mathbb S_n})^*(\Xc) & = & f^*\circ \sigma_{\mathbb S_n}(\Xc)
		\eean
	\end{theo}
	
	\begin{proof} { This proof mimics the proof of} \cite[Theorem 2.4]{lewis1995convex}. Let
		\begin{align*}
		\Xc & = \Sc \times_1 U \cdots \times_D U
		\end{align*}
		denote the Higher Order singular value decomposition of $\Xc$.
		By definition of conjugacy, we have
		\bean
		(f\circ\sigma_{\mathbb S_n})^*(\Xc)
		& = & \sup_{\Yc\in\Rb^{n\times \cdots\times n}} \quad \langle \Xc, \Yc \rangle - f(\sigma_{\mathbb S_n}(\Yc)).
		\eean
		By the tensor von Neumann inequality we have
		%
		\bea
		& & \sup_{\Yc\in\Rb^{n\times \cdots \times n}}\quad \langle \Xc, \Yc \rangle -  f(\sigma_{\mathbb S_n}(\Yc)) \nonumber \\
		& & \hspace{3cm} \le \sup_{\Yc\in\Rb^{n\times \cdots \times n}} \quad
		\frac{1}{D}\sum_{d=1}^D \langle \sigma^{(d)}(\Xc), \sigma^{(d)}(\Yc)\rangle - f(\sigma_{\mathbb S_n}(\Yc)) \label{ystar} \\
		\label{tilt} &  & \hspace{3cm} \le \sup_{s_1,\ldots,s_D\in\Rb^n}
		\quad \frac{1}{\sqrt{D}}\sum_{d=1}^D \langle \sigma^{(d)}(\Sc_\Xc), s_d\rangle-
		f\Bigg(\frac1{\sqrt{D}} (s_{1},\ldots, s_{D})\Bigg).
		\eea
		Notice that the maximizer $s^*$ in the right hand side term of this last equation satisfies $s^*\ge 0$ and $s^*_1=\ldots=s^*_D$ by the symmetry
		of $\Xc$ and $f$.
		Now, based on Section \ref{image}, there
		exists a tensor $\Sc^*$, whose support can clearly be constrained
		to be included in $S$, such that $s^*=\sigma(\Sc^*)$. Thus, we obtain that
		\begin{align}
		&  \sup_{s_1,\ldots,s_D\in\Rb^n}
		\quad \frac{1}{\sqrt{D}}\sum_{d=1}^D \langle \sigma^{(d)}(\Sc_\Xc), s_d\rangle
		-f\Bigg(\frac1{\sqrt{D}} (s_{1},\ldots, s_{D})\Bigg) \nonumber \\
		&  \hspace{3cm} = \: \frac{1}{D}\sum_{d=1}^D \langle \sigma^{(d)}(\Sc_\Xc), \sigma^{(d)}(\Sc^*)\rangle
		-f(\sigma_{\mathbb S_n}(\Sc^*)).
		\nonumber
		\end{align}
		On the one hand, using that $\Sc^*$ has support included in $S$ and Theorem \ref{VN3}, we obtain that
		\begin{align}
		&  \sup_{s_1,\ldots,s_D\in\Rb^n}
		\quad \frac{1}{\sqrt{D}}\sum_{d=1}^D \langle \sigma^{(d)}(\Sc_\Xc), s_d\rangle
		-f\Bigg(\frac1{\sqrt{D}}(s_{1},\ldots, s_{D})\Bigg) \nonumber \\
		&  \hspace{3cm} = \: \langle \Xc, \Xc^*\rangle-f(\sigma_{\mathbb S_n}(\Xc^*))
		\nonumber
		\end{align}
		where
		\begin{align*}
		\Xc^* & = \Sc^* \times_1 U \cdots \times_D U.
		\end{align*}
		From this, we deduce that
		\begin{align}
		&  \sup_{s_1,\ldots,s_D\in\Rb^n}
		\quad \frac{1}{\sqrt{D}}\sum_{d=1}^D \langle \sigma^{(d)}(\Sc_\Xc), s_d\rangle
		-f\Bigg(\frac1{\sqrt{D}}(s_{1},\ldots, s_{D})\Bigg) \nonumber \\
		&  \hspace{3cm} = \: \sup_{\Yc\in\Rb^{n\times \cdots \times n}}\quad \langle \Xc, \Yc \rangle -  f(\sigma_{\mathbb S_n}(\Yc))\\
		&  \hspace{3cm} = \: (f\circ \sigma_{\mathbb S_n})^*(\Xc).
		\nonumber
		\end{align}
		On the other hand, using the fact that $\sigma_{\mathbb S_n}(\Sc_\Xc)=\sigma_{\mathbb S_n}(\Xc)$,
		\begin{align}
		&  \sup_{s_1,\ldots,s_D\in\Rb^n}
		\quad \frac{1}{\sqrt{D}}\sum_{d=1}^D \langle \sigma^{(d)}(\Sc_\Xc), s_d\rangle
		-f\Bigg(\frac1{\sqrt{D}}(s_{1},\ldots, s_{D})\Bigg) \nonumber \\
		&  \hspace{3cm} = \sup_{s_1,\ldots,s_D\in\Rb^n} \quad \frac{1}{\sqrt{D}}\sum_{d=1}^D
		\langle \sigma^{(d)}(\Sc_\Xc), s_d\rangle-f\Bigg(\frac1{\sqrt{D}}(s_{1},\ldots, s_{D})\Bigg) \nonumber \\
		&  \hspace{3cm} = \: f^*\circ \sigma_{{\mathbb S}_n}(\Xc).
		\nonumber
		\end{align}
		Therefore,
		\bean
		(f\circ\sigma_{\mathbb S_n})^*(\Xc) & = & f^*\circ \sigma_{\mathbb S_n}(\Xc)
		\eean
		as announced.
	\end{proof}

	\subsubsection{A closed form formula for the subdifferential}
	We now present a closed form formula for the subdifferential of a symmetric function of the spectrum of a symmetric tensor.
	\begin{cor}
		Let $f: \Rb^{n}\times\cdots\times\Rb^n \mapsto \Rb$ be a symmetric function, i.e.
		\begin{align}
		\label{symf}
		f(s_1,\ldots,s_D) & = f(s_{\tau(1)},\ldots,s_{\tau(D)})
		\end{align}
		for all $\tau \in \mathfrak{S}_S$.
		Then necessary and sufficient conditions for an
		symmetric tensor $\Yc$ to belong to $\partial (f\circ\sigma_{\mathbb R^{nD}})(\Xc)$ are
		\begin{enumerate}
			\item $\Yc$ has the same mode-$d$ singular spaces as $\Xc$ for all $d=1,\ldots,D$
			\item $\sigma_{\mathbb R^{nD}}(\Yc)\in\partial f(\sigma_{\mathbb R^{nD}}(\Xc))$.
		\end{enumerate}
	\end{cor}
	\begin{proof}
		Combine Theorem \ref{conjugate1} with Theorem \ref{119a}.
	\end{proof}
	
	\subsection{The {\em odeco} case}
	Throughout this subsection $\mathbb E=\mathbb R^{nD}$.
	\subsubsection{Proving \eqref{Conjformula} for {\em odeco} tensors}
	As for the symmetric case, we start with a result in the spirit of \eqref{Conjformula}.
	\begin{theo} \label{conjugate2}
		Let $f: \Rb^{n}\times\cdots\times\Rb^n \mapsto \Rb$ satisty property
		\begin{align}
		\label{symff}
		f(s_1,\ldots,s_D) & = f(s_{\tau(1)},\ldots,s_{\tau(D)})
		\end{align}
		for all $\tau \in \mathfrak{S}_S$.
		Then for all {\em odeco} tensors $\Xc$, we have
		\bea\label{114a}
		(f\circ\sigma_{\mathbb R^{nD}})^* (\Xc)= f^*(\sigma_{\mathbb R^{nD}}(\Xc))
		\eea
	\end{theo}
	\begin{proof}
		By definition, equality (\ref{114a}) is equivalent to
		\begin{align}
		\sup_{\Yc}\{\langle \Xc, \Yc \rangle -  f(\sigma_{\mathbb R^{nD}}(\Yc))\} & =\sup_{s_1,\ldots,s_D\in\Rb^n}\left\{\frac{1}{D}\sum_{d=1}^D\langle \sigma^{(d)}(\Xc), s_d\rangle -
		f(s_1,\ldots,s_D) \right\}.
		\label{114g}
		\end{align}
		Consider the optimization problem
		\bea\label{114b}
		\sup_{\Yc}\Big\{\langle \Xc, \Yc \rangle,\quad  f(\sigma_{\mathbb R^{nD}}(\Yc))\leq C\Big\}
		\eea
		and
		\bea \label{114c}
		\sup_{s_1,\ldots,s_D\in\Rb^n}\left\{\frac{1}{D}\sum_{d=1}^D\langle \sigma^{(d)}(\Xc), s_d\rangle,\quad
		f(s_1,\ldots,s_D)\leq C \right\}.
		\eea
		Clearly, we have
		\bea
		&&\sup_{\Yc}\{\langle \Xc, \Yc \rangle,\quad  f(\sigma_{\mathbb R^{nD}}(\Yc))\leq C\} \nonumber \\
		&\leq& \sup_{\Yc}\left\{\frac{1}{D}\sum_{d=1}^D\langle \sigma^{(d)}(\Xc), \sigma^{(d)}(\Yc) \rangle,\quad  f(\sigma_{\mathbb R^{nD}}(\Yc))\leq C\right\} \nonumber \\
		&\leq& \sup_{s_1,\ldots,s_D\in\Rb^n}\left\{\frac{1}{D}\sum_{d=1}^D\langle \sigma^{(d)}(\Xc), s_d\rangle,\quad
		f(s_1,\ldots,s_D)\leq C \right\}. \label{114e}
		\eea
		Assume that the supremum (\ref{114c}) is achieved at $(s^*_1,\ldots,s^*_D)$. Denote
		\bean
		s^{\dagger} = \frac{1}{D!} \sum_{\tau} s_{\tau(k)}.
		\eean
		Clearly, $s^{\dagger}$ is independent of $k$. Moreover, we have
		\bean
		\frac{1}{D}\sum_{d=1}^D\langle \sigma^{(d)}(\Xc), s_d\rangle
		&=& \frac{1}{D}\sum_{d=1}^D\langle \sigma^{(d)}(\Xc), s^\dagger\rangle \\
		f(s^\dagger,\ldots,s^\dagger)
		&\leq& \frac{1}{D!}\sum_{\tau}f(s_{\tau(1)},\ldots,s_{\tau(D)}).
		\eean
		Using (\ref{symf}), we have
		\bean
		\frac{1}{D}\sum_{d=1}^D\langle \sigma^{(d)}(\Xc), s_d\rangle
		& = & \frac{1}{D}\sum_{d=1}^D\langle \sigma^{(d)}(\Xc), s^\dagger\rangle \\
		f(s^\dagger,\ldots,s^\dagger)
		& \leq & C.
		\eean
		This means that the supremum of (\ref{114c}) can also be achieved at
		$(s^\dagger,\ldots,s^\dagger)$.
		Now take an {\em odeco} tensor $\Yc^\dagger$ such that $\sigma^{(d)}(\Yc^\dagger)=s^\dagger$ and $\Yc^\dagger$ has the same singular matrices as $\Xc$. For this particular $\Yc^\dagger$, we have by the equality condition of the generalized von Neumann's Theorem
		\bean
		\langle \Xc, \Yc^\dagger \rangle = \langle \sigma_{\mathbb R^{nD}}(\Xc), \sigma_{\mathbb R^{nD}}(\Yc^\dagger) \rangle
		=\frac{1}{D}\sum_{d=1}^D\langle \sigma^{(d)}(\Xc), s^\dagger \rangle
		\eean
		and
		\bean
		f(\sigma_{\mathbb R^{nD}}(\Yc^\dagger)) = f(s^\dagger,\ldots,s^\dagger) \leq C.
		\eean
		We then deduce that
		\bea
		&&\sup_{\Yc}\{\langle \Xc, \Yc \rangle,\quad  f(\sigma_{\mathbb R^{nD}}(\Yc))\leq C\}  \geq \nonumber \\
		&& \hspace{1cm} \langle \Xc, \Yc^\dagger \rangle \,\,\, = \sup_{s_1,\ldots,s_D\in\Rb^n}\left\{\frac{1}{D}\sum_{d=1}^D\langle \sigma^{(d)}(\Xc), s_d\rangle,\quad
		f(s_1,\ldots,s_D)\leq C \right\}. \label{114f}
		\eea
		Combining (\ref{114e}) and (\ref{114f}) gives
		\bean
		&&\sup_{\Yc}\{\langle \Xc, \Yc \rangle,\quad  f(\sigma_{\mathbb R^{nD}}(\Yc))\leq C\} = \\
		&& \hspace{1cm} \sup_{s_1,\ldots,s_D\in\Rb^n}\left\{\frac{1}{D}\sum_{d=1}^D\langle \sigma^{(d)}(\Xc), s_d\rangle,\quad
		f(s_1,\ldots,s_D)\leq C \right\}.
		\eean
		Then (\ref{114g}) follows.
	\end{proof}
	
	\subsubsection{A closed form formula for a subset of the subdifferential}	
	The following Corollary is a direct consequence of Theorem \ref{119a} and Theorem \ref{conjugate2}.
	\begin{coro}
		Let $f: \Rb^{n}\times\cdots\times\Rb^n \mapsto \Rb$ satisty property
		\begin{align}
		\label{symfff}
		f(s_1,\ldots,s_D) & = f(s_{\tau(1)},\ldots,s_{\tau(D)})
		\end{align}
		for all $\tau \in \mathfrak{S}_S$.
		Let $\Xc$ be an {\em odeco} tensor. Then necessary and sufficient conditions for an
		{\em odeco} tensor $\Yc$ to belong to $\partial (f\circ\sigma_{\mathbb R^{nD}})(\Xc)$ are
		\begin{enumerate}
			\item $\Yc$ has the same mode-$d$ singular spaces as $\Xc$ for all $d=1,\ldots,D$
			\item $\sigma_{\mathbb R^{nD}}(\Yc)\in\partial f(\sigma_{\mathbb R^{nD}}(\Xc))$.
		\end{enumerate}
		Moreover, the closure of the convex combination of these {\em odeco} tensors
		belongs to $\partial (f\circ\sigma_{\mathbb R^{nD}})(\Xc)$.
	\end{coro}

	\section{The subdifferential of tensor Schatten norms for symmetric and {\em odeco} tensors}
	\label{applischatten}
	
	In this section, we compute the subdifferential of the Schatten norm \eqref{schattenpq612} for symmetric and {\em odeco} tensors. 
	Consider $f$: $\mathbb R^{n}\times \cdots \times \mathbb R^{n} \mapsto \mathbb R$ defined by
	\begin{align}
	\label{schatt}
	f(s_1,\ldots,s_D) & = \lambda\Big(\sum_{d=1}^D \|s_d\|_p^q\Big)^{1/q}
	\end{align}
	for some integers $p,q\geq 1$ and constant $\lambda>0$. Clearly, $f$ is a convex function and 
	\bean
	\|\cdot\|_{p,q}= f \Big( \sigma^{(1)}(\cdot),\ldots,\sigma^{(D)}(\cdot)\Big).
	\eean

	\subsection{The case $p,q>1$}
In this case we can write (\ref{schatt}) as
	\begin{align}
	f(s_1,\ldots,s_D) & =
	\lambda\!\!\!\!\! \sup_{\Vert w\Vert_{q^*}=1 \atop  \Vert v_d\Vert_{p^*}=1, d=1,\ldots,D}
	 \ \left\{
	\sum_{d=1}^D \ w_d \ \langle v_d,s_d\rangle \right\}.
	\label{variaschatt}
	\end{align}
	Notice that since the supremum in \eqref{variaschatt} is taken over a compact set and the function to be maximized is continuous, then this supremum is attained. Let $\mathcal V\mathcal W^*$ denote the set of maximizers in the variational formulation of $f$ \eqref{variaschatt}. Then, by
	\cite{hiriart2013convex} the subdifferential of $f$ is given by
	\begin{align*}
	\partial f(v_1,\ldots,v_D) & =\lambda\,\,
	\overline{\rm conv} \Big\{
	(w^*_1 v^*_1,\ldots, w^*_D v^*_D) \mid (v_1^*,\ldots,v_D^*,w^*) \in \mathcal VW^* \Big\}.
	\end{align*}
	Notice that $\mathcal V\mathcal W^*$ is fully characterized by
	\begin{align*}
	v^*_d & =
	\begin{cases}
	s_d / \Vert s_d\Vert_{p^*} \textrm{ if } s_d \neq 0 \\
	\\
	B_{p^*} \textrm{ otherwise }
	\end{cases} \\
	w^* & =
	\begin{cases}
	\omega^* / \Vert \omega^* \vert_{q^*} \textrm{ if } \omega^* \neq 0 \\
	\\
	B_{q^*} \textrm{ otherwise}.
	\end{cases}
	\end{align*}
	with
	\begin{align*}
	\omega^* & = (\langle v^*_d,s_d\rangle)_{d=1}^D
	\end{align*}
	and $B_p$ denotes the unit ball in the $\ell_p$ norm.
	
	Using these computations, we obtain the following result.
	\begin{theo}
		We have
		\begin{enumerate}
			\item the subdifferential of the nuclear norm for symmetric tensors
			is the set of tensors $\Yc$ satisfying
			\begin{enumerate}
				\item $\Yc$ has the same mode-$d$ singular spaces as $\Xc$ for all $d=1,\ldots,D$
				\item $\sigma_{\mathbb R^{nD}}(\Yc)_{d}=w^*_d v^*_d$ if $\sigma^{(d)}(\Xc)\neq 0$,
				\item $\sigma_{\mathbb R^{nD}}(\Yc)_{d}=0$ if $\sigma^{(d)}(\Xc)=0$ and if $\sigma^{(d')}(\Xc)\neq 0$ for some $d'$.
				\item $\sigma_{\mathbb R^{nD}}(\Yc)_{d} \in w_d^* B_{p^*}$, $w^* \in B_{q^*}$ if $\sigma^{(d')}(\Xc)=0$ for all $d'=1,\ldots,D$.
			\end{enumerate}
			
			\item the subdifferential of the nuclear norm for {\em odeco} tensors
			contains the closure of the convex hull of {\em odeco} tensors $\Yc$ satisfying
			\begin{enumerate}
				\item $\Yc$ has the same mode-$d$ singular spaces as $\Xc$ for all $d=1,\ldots,D$
				\item $\sigma_{\mathbb R^{nD}}(\Yc)_{d}=w^*_d v^*_d$ if $\sigma^{(d)}(\Xc)\neq 0$,
				\item $\sigma_{\mathbb R^{nD}}(\Yc)_{d}=0$ if $\sigma^{(d)}(\Xc)=0$ and if $\sigma^{(d')}(\Xc)\neq 0$ for some $d'$.
				\item $\sigma_{\mathbb R^{nD}}(\Yc)_{d} \in w_d^* B_{p^*}$, $w^* \in B_{q^*}$ if $\sigma^{(d')}(\Xc)=0$ for all $d'=1,\ldots,D$.
			\end{enumerate}
		\end{enumerate}
	\end{theo}
	
	\subsection{The nuclear norm}
	Consider $f(\cdot):\Rb^n\times\cdots\times\Rb^n \to \Rb$ defined by
	\bean
	f(s_1,\ldots,s_D) = \frac{1}{D}\sum_{i=1}^D\|s_i\|_1.
	\eean
	Then  for any $(s_1,\ldots,s_D)\in \Rb^n\times\cdots\times\Rb^n$, we have
	\bean
	\partial f(s_1,\ldots,s_D)=\frac{1}{D}\{(c_1,\ldots,c_D) \},
	\eean
	where $c_d = (c_{d1},\ldots,c_{dn})^t$ for $d=1,\ldots,D$ satisfies
	\bean
	c_{dj} = \left\{ \begin{array}{ll}1 & s_{dj}>0 \\ -1 & s_{dj} <0 \\ \omega_{dj}  & s_{dj}=0
	\end{array} \right.
	\eean
	with $\omega_{ij}$ being any real number in the interval $[-1 , 1]$.
	
	Thus, we obtain the following result.
	
	\begin{theo}
		We have that
		\begin{enumerate}
			\item the subdifferential of the nuclear norm for symmetric tensors
			is the set of tensors $\Yc$ satisfying
			\begin{enumerate}
				\item $\Yc$ has the same mode-$d$ singular spaces as $\Xc$ for all $d=1,\ldots,D$
				\item $\sigma_{\mathbb R^{nD}}(\Yc)_{dj}=1$ if $\sigma^{(d)}_j(\Xc)>0$,
				\item $\sigma_{\mathbb R^{nD}}(\Yc)_{dj}\in [0,1]$ if $\sigma^{(d)}_j(\Xc)=0$.
			\end{enumerate}
			
			\item the subdifferential of the nuclear norm for {\em odeco} tensors
			contains the closure of the convex hull of {\em odeco} tensors $\Yc$ satisfying
			\begin{enumerate}
				\item $\Yc$ has the same mode-$d$ singular spaces as $\Xc$ for all $d=1,\ldots,D$
				\item $\sigma_{\mathbb R^{nD}}(\Yc)_{dj}=1$ if $\sigma^{(d)}_j(\Xc)>0$,
				\item $\sigma_{\mathbb R^{nD}}(\Yc)_{dj}\in [0,1]$ if $\sigma^{(d)}_j(\Xc)=0$.
			\end{enumerate}
		\end{enumerate}
	\end{theo}
	
	\subsection{Remark on the remaining cases}
	We leave to the reader the easy task of deriving the general formulas for the cases $p=1$ and $q>1$, $p>1$ and $q=1$.
	
	\section{Conclusion and perspectives}
	
	In this paper, we studied the subdifferential of some tensor norms for symmetric tensors and {\em odeco} tensors. We provided a complete characterization of for the symmetric case and described a subset of the subdifferential for {\em odeco} tensors. The main tool in our analysis is an extension of the Von Neumann's trace inequality to the tensor setting recently proved in \cite{chretien2015neumann}.
	Such results may find applications in the field of Compressed Sensing. A lot of work remains in order to extend our results to non-symmetric settings. We plan to investigate this question in a future research project.
	
	
	
	

	\bibliographystyle{amsplain}
	\bibliography{TensorSubdiff}

\end{document}